\documentclass[numbook,envcountsame]{svjour3}

\usepackage{dpr_fec}
\allowdisplaybreaks

\newcommand{\etasuccessful}{\eta}

\newcommand{\TRk}{\delta_{k}}
\newcommand{\TRnext}{\delta_{k+1}}
\newcommand{\norm}[1]{\|#1\|}
\newcommand{\s}[1]{^{#1}}
\newcommand{\xnext}{x_{k+1}}

\newcommand{\eps}{\epsilon}
\newcommand{\gammaLB}{\underline\gamma}
\newcommand{\gammaUB}{\bar\gamma}

\newcommand{\gammamin}{\gamma_{\min}}
\newtheorem{assumption}[theorem]{Assumption}
\newtheorem{TRchoice}[theorem]{Update}

\smartqed

\title{Concise Complexity Analyses for Trust-Region Methods}
\titlerunning{Concise Complexity Analyses for Trust-Region Methods}
\author{\makebox[\linewidth][l]{Frank E. Curtis \and Zachary Lubberts \and Daniel P. Robinson}}
\authorrunning{F. E. Curtis and Z. Lubberts and D.~P.~Robinson}
\institute{
Frank E. Curtis \at Dept. of Industrial and Systems Engineering, Lehigh University, Bethlehem, PA, USA. \\
\email{\href{mailto:frank.e.curtis@gmail.com}{frank.e.curtis@gmail.com}} 
\and
Zachary Lubberts \at Dept. of Applied Mathematics and Statistics, Johns Hopkins University, Baltimore, MD, USA. \\ \email{\href{mailto:zlubber1@jhu.edu }{zlubber1@jhu.edu}} 
\and
Daniel P. Robinson \at Dept. of Applied Mathematics and Statistics, Johns Hopkins University, Baltimore, MD, USA. \\  \email{\href{mailto:daniel.p.robinson@gmail.com}{daniel.p.robinson@gmail.com}}   
}
\date{\today}

\begin{document}

\maketitle

\begin{abstract}
Concise complexity analyses are presented for simple trust region algorithms for solving unconstrained optimization problems. In contrast to a traditional trust region algorithm, the algorithms considered in this paper require certain control over the choice of trust region radius after any successful iteration.  The analyses highlight the essential algorithm components required to obtain certain complexity bounds.  In addition, a new update strategy for the trust region radius is proposed that offers a second-order complexity bound.



\end{abstract}

\section{Introduction}\label{sec.introduction}

We analyze a trust region framework for solving the smooth optimization problem
\begin{equation} \label{prob}
\min_{x\in\R{n}} \ f(x),
\end{equation}
where $f:\R{n} \to \R{}$.  Since trust region methods have been extensively studied and analyzed, let us immediately discuss the contributions of this paper. 

\subsection{Contributions}

We show in Section~\ref{sec.first} that explicitly connecting the trust region radius with the norm of the gradient of $f$ allows for a concise first-order complexity analysis of a trust region method. This method is, in fact, a specific instance of the general framework considered in~\cite{FanYuan01,grapiglia2015convergence,toint2013nonlinear}.  Our analysis, however, is simpler since we do not consider such a general framework.  In addition, to the best of our knowledge, our presentation is the first to highlight the essential aspects needed to obtain a first-order complexity bound. In Section~\ref{sec.second}, we propose a new update strategy for the trust region radius that allows one to obtain a second-order complexity bound.  In Section~\ref{sec.ye}, we contrast the aforementioned strategies with one that offers better complexity bounds, but has some practical disadvantages.

\subsection{Notation and Assumption about the Objective Function}

We use $\R{}$ to denote the set of real numbers, $\R{}_{+}$ (respectively $\R{}_{++}$) to denote the set of nonnegative (respectively positive) real numbers, $\R{n}$ to denote the set of $n$-dimensional real vectors, and $\R{m \times n}$ to denote the set of $m$-by-$n$-dimensional real matrices.  The set of natural numbers is denoted as $\N{} := \{0,1,2,\dots\}$.  

We denote $g := \nabla f : \R{n} \to \R{}$ and $H := \nabla^2 f : \R{n} \to \R{n \times n}$.  For all $k\in\N{}$, we let $x_k$ denote the $k$th iteration computed by the trust region method. For brevity, we append $k \in \N{}$ as a subscript to a function to denote its value at the $k$th iterate $x_k$, e.g., $f_k = f(x_k)$, $g_k = g(x_k)$, and $H_k := H(x_k)$. Given $H_k$, we let  $\lambda_k := \lambda_k(H_k)$ denote the leftmost eigenvalue of $H_k$, and $(\lambda_k)_- := \min\{0,\lambda_k\}$, i.e., $(\lambda_k)_-$ is the negative part of $\lambda_k$. Finally, for $v\in\R{n}$, we use $\|v\|$ to denote the two-norm of $v$.

The following assumption on the objective function is made throughout.
\begin{assumption}\label{ass.main}
The function $f : \R{n} \to \R{}$ is twice continuously differentiable with Hessian function $H$ being Lipschitz continuous, i.e., there exists a constant $L\in\R{}_{++}$ such that $\|H(x) - H(y)\| \leq L \|x-y\|$ for all $(x,y)\in\R{n}\times\R{n}$. In addition, the function~$f$ is bounded below, 
i.e., there exists $f_{\inf} \in \R{}$ such that $f(x) \geq f_{\inf}$ for all $x\in\R{n}$.
\end{assumption}

\section{A Trust Region Algorithm} 

In this section, we present and analyze trust region algorithms, the framework for which is formally stated in Section~\ref{sec.alg.1}. The key results needed to perform most of our complexity analyses are stated and proved in Section~\ref{sec.analysis.1}. Using these key results, we establish a first-order complexity result in Section~\ref{sec.first} and a second-order complexity result in Section~\ref{sec.second}.  Specific strategies employed to obtain these complexity results are stated in each subsection.  We end by considering an instance with a fixed trust region radius in Section~\ref{sec.ye} that has advantages and disadvantages vis-\`a-vis these other strategies.

\subsection{The Algorithm}\label{sec.alg.1}

We study the trust region method stated as Algorithm~\ref{alg-gtr}.  During the $k$th iteration, given a trust region radius $\delta_k\in\R{}_{++}$, the algorithm computes an approximate solution $s_k\in\R{n}$ to the trust region subproblem
\begin{equation}\label{def.Bmod}
\min_{s\in\R{n}} \ \left\{ m_k(s) := f_k + g_k^T s + \thalf s^T H_k s \right\} \ \ 
\st \ \ \|s\| \leq \delta_k.
\end{equation}
To facilitate a unified first- and second-order complexity analysis, when $\lambda_k < 0$ we allow in Step~\ref{step.delta.1} the radius to be set as either $\TRk \gets \gamma_k \|g_k\|$ or $\TRk \gets \gamma_k|(\lambda_k)_-|$. 
For this reason, it will be convenient for our analyses to refer to the index sets
\begin{align*}
\Kcal_g &:= \{k\in\N{}: \ \delta_k \gets \gamma_k \|g_k\| \ \text{in either Step~\ref{step.delta.1} or Step~\ref{step.delta.2}}\} \\ 
\text{and} \ \ \Kcal_H &:= \{k\in\N{}: \ \delta_k \gets \gamma_k |(\lambda_k)_-| \ \text{in Step~\ref{step.delta.1}}\}.
\end{align*}
Note that $\lambda_k < 0$ for all $k\in\Kcal_H$ due to Step~\ref{step.lam.neg}.
We define an approximate solution to~\eqref{def.Bmod} in terms of a Cauchy point that applies to our setting,
i.e., a vector that ensures that the model $m_k$ is sufficiently reduced. Specifically, if $u_k$
denotes a unit eigenvector corresponding to $\lambda_k$ scaled by $\pm1$ so that $g_k^Tu_k \leq 0$, then with
\begin{equation}\label{def.v}
v_k := 
\begin{cases}
-g_k & \text{if $k\in\Kcal_g$} \\
\phantom{-} u_k & \text{if $k\in\Kcal_H$,}
\end{cases}
\end{equation}
the Cauchy point $s_k\s{c}$ is defined as
\begin{equation}\label{new.cauchy}
s_k\s{c} := t_k v_k,
 \ \text{where} \
t_k := 
\argmin{t\geq 0} \ m_k(tv_k) \ \st\ \|tv_k\| \leq \delta_k.
\end{equation} 
We say that any $s_k$ satisfying $m_k(s_k) \leq m_k(s_k^c)$ and 
$\|s_k\| \leq \delta_k$ is a valid choice for an approximate solution to problem~\eqref{def.Bmod}.

\begin{algorithm}[b]
  \caption{Trust region algorithm}
  \label{alg-gtr}
  \begin{algorithmic}[1]
    \State Input an initial estimate $x_0\in\R{n}$ of a solution to~\eqref{prob}.
    \State Choose parameters $(\gamma_c,\etasuccessful) \in (0,1) \times (0,1)$ and $0 < \gammaLB \leq \gammaUB < \infty$. 
    \State Choose $\gamma_0 \in[\gammaLB,\gammaUB]$.
    \For{$k = 0,1,2,\dots$}
       \If{$\lambda_k < 0$} \label{step.lam.neg}
          \State Set either $\TRk \gets \gamma_k \|g_k\|$ or $\TRk \gets \gamma_k|(\lambda_k)_-|$. \label{step.delta.1}
       \Else
          \State Set $\TRk \gets \gamma_k \|g_k\|$. \label{step.delta.2}
       \EndIf
       \State Find any trial step $s_k$ that satisfies $\norm{s_k} \leq \TRk$ and $m_k(s_k) \leq m_k(s_k\s{c})$. \label{step.requirements}
       \State Set\label{step:rho}
       $$
       \rho_k
       \gets  \frac{f_k - f(x_k + s_k)}{m_k(0) - m_k(s_k)}.
       $$
       \If{$\rho_k \geq \etasuccessful$,}
       \State\label{TR-inc-vs} Set $\xnext \gets x_k + s_k$ and choose any $\gamma_{k+1} \in [\gammaLB,\gammaUB]$.
       $\Comment{\text{successful}}$
      \Else
        \State Set $\xnext \gets x_k$ and $\gamma_{k+1} \gets \gamma_c \gamma_k$. $\Comment{\text{unsuccessful}}$ \label{step.U}
      \EndIf
    \EndFor
  \end{algorithmic}
\end{algorithm}

The updates for setting $x_{k+1}$ and $\gamma_{k+1}$ depend on the ratio of the decrease in~$f$ to the decrease predicted by the model $m_k$, as denoted by $\rho_k$ in Step~\ref{step:rho}. If $\rho_k$ is larger than a pre-specified value $\etasuccessful\in(0,1)$, then $x_{k+1} \gets x_k + s_k$ and any value for $\gamma_{k+1}$ satisfying $\gamma_{k+1}  \in [\gammaLB,\gammaUB]$ with $0 < \gammaLB \leq \gammaUB < \infty$ may be used.  In this event, the iteration is said to be \emph{successful}. On the other hand, if $\rho_k < \etasuccessful$, then $x_{k+1} \gets x_k$ and $\gamma_{k+1} \gets \gamma_c \gamma_k$ with $\gamma_c \in (0,1)$; the iteration is said to be \emph{unsuccessful}. It will be helpful for the analysis to define the sets of successful and unsuccessful iterations: 
\begin{equation} \label{def.SandU}
\Scal := \{k\in\N{}: \rho_k \geq \eta\} 
\ \ \text{and} \ \ 
\Ucal := \{k\in\N{}: \rho_k < \eta\}. 
\end{equation}

We now turn to the key results that are needed in our analyses.

\subsection{Key Results Needed for Complexity Analyses}\label{sec.analysis.1}

We start by making the following assumption on the iterates.
\begin{assumption}\label{ass.main.H}
The Hessian function $H$ is uniformly bounded over the sequence of iterates, i.e., for some $\kappa\in\R{}_{+}$ and all $k\in\N{}$, it holds that $\|H_k\| \leq \kappa$.
\end{assumption}

Note that Assumption~\ref{ass.main.H} is implied by Assumption~\ref{ass.main} any time $\{x_k\}$ is contained in a bounded set, which is a common assumption used in some analyses.

Our first result gives the decrease in $m_k$ guaranteed by the Cauchy point.

\begin{lemma} \label{lem-achievable}
For all $k \in \N{}$, the trial step $s_k$ satisfies
\begin{equation} \label{lem.mod.dec}
m_k(0) - m_k(s_k)
\geq
\begin{cases}
\thalf\min\big\{(1+\kappa)^{-1},\gamma_k\big\} \|g_k\|^2 & \text{if $k\in\Kcal_g$,} \\
\thalf \gamma_k^2 |(\lambda_k)_-|^3 & \text{if $k\in\Kcal_H$.}
\end{cases}
\end{equation}
\end{lemma}
\begin{proof}
First, suppose that $k\in\Kcal_g$, which implies that $v_k = -g_k$ in the definition of the Cauchy point in~\eqref{new.cauchy}. In this case, it follows from the decrease guaranteed by the Cauchy point~\cite[Theorem~6.3.1]{ConGT00a}, $\TRk \gets \gamma_k\|g_k\|$, and Assumption~\ref{ass.main.H} that
\begin{equation*} 
m_k(0) - m_k(s_k^c)
\geq \thalf \| g_k \| \min\left\{ \frac{\|g_k\|}{1
 + \|H_k\|}, \TRk \right\}
\geq \thalf\min\left\{\frac{1}{1+\kappa},\gamma_k\right\} \|g_k\|^2.
\end{equation*}
The result~\eqref{lem.mod.dec} follows from this fact and $m_k(s_k) \leq m_k(s_k^c)$, as required in Step~\ref{step.requirements}. 

Second, suppose that $k\in\Kcal_H$, which implies that $v_k = u_k$ in the definition of the Cauchy point in~\eqref{new.cauchy}; recall that $u_k$ denotes a unit eigenvector of $H_k$ corresponding to $\lambda_k$ scaled by $\pm1$ so that $g_k^Tu_k \leq 0$.
By combining this with $k\in\Kcal_H$ so that $\delta_k \gets \gamma_k|(\lambda_k)_-|$ with $\lambda_k < 0$, \eqref{def.v}, and~\eqref{new.cauchy} we obtain
\bequationNN
  \baligned
    &\ \min_{t\geq0}\ m_k(t u_k)\ \st\ \|t u_k\|_2 \leq \delta_k = \gamma_k |(\lambda_k)_-| \\
    =&\ \min_{t\geq0}\ f_k + g_k^T(t u_k) + \thalf (t u_k)^T H_k(t u_k)\ \st\ t \leq \gamma_k |(\lambda_k)_-| \\
    =&\ \min_{t\geq0}\ f_k + t g_k^Tu_k + \thalf t^2 \lambda_k\ \st\ t \leq \gamma_k |(\lambda_k)_-|.
  \ealigned
\eequationNN
Since $g_k^Tu_k \leq 0$ and $\lambda_k < 0$, the minimum occurs at $t_k = \gamma_k |(\lambda_k)_-|$, which combined with $m_k(x_k) \leq m_k(s_k\s{c})$ in Step~\ref{step.requirements} of Algorithm~\ref{alg-gtr} yields
\begin{align*}
m_k(s_k)
&\leq m_k(s_k\s{c}) \\
&= \min_{t\geq0}\ m_k(t u_k)\ \st\ \|t u_k\|_2 \leq \delta_k = \gamma_k |(\lambda_k)_-| \\ 
&= m_k(t_ku_k) 
      = f_k - \gamma_k\lambda_k g_k^Tu_k - \thalf \gamma_k^2 |(\lambda_k)_-|^3 \\
      &\leq f_k - \thalf \gamma_k^2|(\lambda_k)_-|^3
      = m_k(0) - \thalf \gamma_k^2|(\lambda_k)_-|^3.
\end{align*}
Hence, the reduction in $m_k$ obtained by $s_k$ is bounded as in~\eqref{lem.mod.dec}. \qed
\end{proof}

We now bound the difference between the objective function and its model.

\begin{lemma} \label{lem-difference}
For all $k\in\N{}$, the error in the model $m_k$ at $s_k$  can be bounded as
$$
|\,f(x_k^{ } +s_k^{ }) - m_k^{ }(s_k^{ })| 
\leq 
\begin{cases}
\kappa \gamma_k^2 \|g_k\|^2 & \text{if $k\in\Kcal_g$,} \\
\tfrac16 L \gamma_k^3|(\lambda_k)_-|^3 & \text{if $k\in\Kcal_H$.}
\end{cases}
$$
\end{lemma}
\begin{proof}
If $k\in\Kcal_g$, the result follows from~\cite[Theorem~6.4.1]{ConGT00a}, Assumption~\ref{ass.main.H}, and the fact that $\TRk \gets \gamma_k\|g_k\|$ for $k\in\Kcal_g$. 
If $k\in\Kcal_H$, the result follows from Assumption~\ref{ass.main},  \cite[Theorem~3.1.5]{ConGT00a}, and the fact that $\TRk \gets \gamma_k|(\lambda_k)_-|$ for $k\in\Kcal_H$. \qed
\end{proof}

We can now give a uniform lower bound on $\{\gamma_k\}$ that is independent of $k$.

\begin{lemma} \label{lem-progress}
For all $k\in\N{}$, it holds that
\begin{equation} \label{gamma.lb}
\gamma_k \geq \gammamin := \min\left\{\gammaLB, \frac{\gamma_c}{1+\kappa}, \frac{\gamma_c(1-\etasuccessful)}{2\kappa}, \frac{3\gamma_c(1-\eta)}{L}\right\} \in (0,1).
\end{equation}
\end{lemma}
\begin{proof}
For a proof by induction, 
we first note that $\gamma_0 \geq \gammaLB \geq \gammamin$ by the choice for $\gamma_0$  in Algorithm~\ref{alg-gtr}, so that~\eqref{gamma.lb} holds when $k = 0$. Next, supposing that~\eqref{gamma.lb} holds for $k$, we proceed to show that it also holds with $k$ replaced by $k+1$.
\benumerate
  \item[\underline{Case 1:}] $\gamma_k
> \min\left\{1/(1+\kappa), (1-\etasuccessful)/(2\kappa),3(1-\eta)/L\right\}$. In this case, it holds that
$$
\gamma_{k+1}
\geq \min\{\gammaLB,\gamma_c\gamma_k\}
\geq \min\left\{\gammaLB,\gamma_c/(1+\kappa), \gamma_c(1-\etasuccessful)/(2\kappa), 3\gamma_c(1-\eta)/L\right\}
\equiv \gammamin,
$$
meaning that~\eqref{gamma.lb} holds with $k$ replaced by $k+1$ as claimed.
  \item[\underline{Case 2:}] $\gamma_k
\leq \min\left\{1/(1+\kappa), (1-\etasuccessful)/(2\kappa), 3(1-\eta)/L\right\}$.  We must consider two subcases.  First, suppose that $k\in\Kcal_g$. It follows from the definition of $\rho_k$, Lemma~\ref{lem-difference}, and Lemma~\ref{lem-achievable}
that
\begin{align*}
| \rho_k - 1 |
&= \frac{|\,f(x_k+s_k) - m_k(s_k)|}{m_k(0) - m_k(s_k)} 
\leq \frac{2\kappa \gamma_k^2 \|g_k\|^2}{\min\big\{(1+\kappa)^{-1},\gamma_k\big\} \|g_k\|^2}
= 2\kappa\gamma_k
\leq 1-\eta,
\end{align*}
which implies that $\rho_k \geq \etasuccessful$, i.e., that $k\in\Scal$.  Second, suppose that $k\in\Kcal_H$.  It follows from Lemma~\ref{lem-difference} and Lemma~\ref{lem-achievable}
that
\begin{align*}
| \rho_k - 1 |
&= \frac{|\,f(x_k+s_k) - m_k(s_k)|}{m_k(0) - m_k(s_k)} 
\leq \frac{L\gamma_k^3|(\lambda_k)_-|^3}{3\gamma_k^2|(\lambda_k)_-|^3}
= \frac{L\gamma_k}{3}
\leq 1-\eta,
\end{align*}
which implies $\rho_k \geq \etasuccessful$, i.e., that $k\in\Scal$.  In either subcase, it follows that $k\in\Scal$.  Using $k\in\Scal$, we have from Algorithm~\ref{alg-gtr} that $\gamma_{k+1} \geq \gammaLB \geq \gammamin$, so that~\eqref{gamma.lb} holds with $k$ replaced by $k+1$ as claimed.
\eenumerate
The result follows since we proved the inductive step in each case.\qed 
\end{proof}

The next result gives a refined bound on the decrease in the model $m_k$ for all $k\in\N{}$, as well as gives a bound on the decrease in $f$ when $k\in\Scal$.

\begin{lemma}\label{lem.dec.m.g}
For all $k\in\N{}$, the trial step $s_k$ satisfies
\begin{equation} \label{mod-dec}
m_k(0) - m_k(s_k)
\geq
\begin{cases}
 \thalf\gammamin \|g_k\|^2 & \text{if $k\in\Kcal_g$,} \\
  \thalf \gammamin^2 |(\lambda_k)_-|^3 & \text{if $k\in\Kcal_H$.} 
\end{cases}
\end{equation}
In addition, with
$\kappa_{\min}:=\thalf\eta\gammamin^2$, we have for all $k \in \N{}$ that
\begin{equation}\label{f-dec}
f_k - f_{k+1}
\geq
\begin{cases}
 \kappa_{\min} \|g_k\|^2 & \text{if $k\in\Kcal_g\cap\Scal$,} \\
 \kappa_{\min} |(\lambda_k)_-|^3 & \text{if $k\in\Kcal_H\cap\Scal$.}
\end{cases}
\end{equation}
\end{lemma}
\begin{proof}
The lower bound~\eqref{mod-dec} follows from Lemma~\ref{lem-achievable} and Lemma~\ref{lem-progress}, after observing that $\gammamin\leq\gamma_c(1+\kappa)^{-1}\leq (1+\kappa)^{-1}$.  The result~\eqref{f-dec} follows from~\eqref{mod-dec} and the fact that $\rho_k \geq \eta$ when $k\in\Scal$. \qed
\end{proof}

We now show that the maximum number of consecutive unsuccessful iterations can be bounded. It is interesting to note that this result does not depend on the specific manner in which $\TRk$ is chosen in Step~\ref{step.delta.1} or Step~\ref{step.delta.2}.

\begin{lemma} \label{lem.bound-consec-U}
The number of consecutive iterations in $\Ucal$ is at most $\lceil\log_{\gamma_c}(\tfrac{\gammamin}{\gammaUB})\rceil > 0$.
\end{lemma}
\begin{proof}
The update strategy for $\{\gamma_k\}$ ensures that $\gamma_k \leq \gammaUB$ for all $k\in\N{}$. Also, it follows from Lemma~\ref{lem-progress} that $\gamma_k \geq \gammamin$ for all $k\in\N{}$.  Since the update $\gamma_{k+1} \gets \gamma_c \gamma_k$ is used when $k\in\Ucal$, we must conclude that the maximum number of consecutive iterations in $\Ucal$ can be no larger than $\lceil\log_{\gamma_c}(\gammamin/\gammaUB)\rceil > 0$, as claimed. \qed
\end{proof}

\subsection{A Strategy with a Concise First-Order Complexity Analysis}\label{sec.first}

In this section, our aim is to provide an upper bound on the maximum number of iterations until the norm of the gradient falls below a threshold value, say $\eps_g\in\R{}_{++}$.  For this reason, it will be convenient to define the sets
$$
\Scal_1(\eps_g) := \{k\in\Scal: \|g_k\| > \eps_g\}\ \ \text{and}\ \ \Kcal_1(\eps_g) := \{k\in\N{}: \|g_k\| > \eps_g\}.
$$
Also, since we are currently interested in approximate \emph{first-order} stationarity, it is reasonable to use the following trust region radius update strategy.
\begin{TRchoice}\label{up.1}
For any $k\in\N{}$ such that Step~\ref{step.delta.1} is reached, we set $\TRk \gets \gamma_k\|g_k\|$. 
\end{TRchoice}
Combining Update~\ref{up.1} with Step~\ref{step.delta.2} shows that $\TRk \gets \gamma_k \|g_k\|$ for all $k\in\N{}$ so that 
\begin{equation} \label{KgisN}
\Kcal_g = \N{}
\ \ \text{and} \ \
\Kcal_H = \emptyset.
\end{equation}
The remaining results of this section assume that Update~\ref{up.1} is used. 

We start by proving an upper bound on the size of the index set $\Scal_1(\eps_g)$.

\begin{lemma}\label{lem.bound-Seps}
For any $\eps_g\in\R{}_{++}$, the size of $\Scal_1(\eps_g)$ satisfies
\begin{equation}\label{Seps-bound}
|\Scal_1(\eps_g)|
\leq \left\lfloor \left(\frac{f_0 - f_{\inf}}{\kappa_{\min}}\right) \eps_g^{-2}\right\rfloor. 
\end{equation}
\end{lemma}
\begin{proof}
We start by noting that $\Scal_1(\eps_g) \subseteq \N{} = \Kcal_g$ because of~\eqref{KgisN}. 
Combining this with Assumption~\ref{ass.main}, monotonicity of $\{f_k\}$, \eqref{f-dec}, and the definition of $\Scal_1(\eps_g)$ gives
\begin{align*}
f_0 - f_{\inf}
&\geq \sum_{k\in\Scal_1(\eps_g)} \!\!\big(f_{k}-f_{k+1}\big)
\geq \kappa_{\min}\sum_{k\in\Scal_1(\eps_g)} \!\!\|g_k\|^2
\geq \kappa_{\min} \eps_g^2 |\Scal_1(\eps_g)|.
\end{align*}
The bound in~\eqref{Seps-bound} follows from this inequality. \qed
\end{proof}

We obtain the complexity result by combining the last with Lemma~\ref{lem.bound-consec-U}.

\begin{theorem}\label{thm-complexity}
For any $\eps_g\in\R{}_{++}$, the size of $\Kcal_1(\eps_g)$ satisfies
$$
|\Kcal_1(\eps_g)| 
\leq
\left\lceil\log_{\gamma_c}\(\frac{\gammamin}{\gammaUB}\)\right\rceil \left\lfloor \left(\frac{f_0 - f_{\inf}}{\kappa_{\min}}\right) \eps_g^{-2}\right\rfloor
= \Ocal(\eps_g^{-2}).
$$
\end{theorem}
\begin{proof}
The result follows by combining Lemma~\ref{lem.bound-Seps} and Lemma~\ref{lem.bound-consec-U}. \qed
\end{proof}

We conclude this section by discussing how the trust region radius update considered in this section compares to a traditional strategy. In fact, the strategies are the same for unsuccessful iterations since they both set $\delta_{k+1} \gets \gamma_c \TRk$. However, if $k$ is a successful iteration, then a traditional strategy sets $\delta_{k+1}^{trad} \gets \max\{\gamma_e\|s_k\|, \TRk\}$ for some $\gamma_e \geq 1$.  This update is also allowed by 
Update~\ref{up.1} as long as
$$
\TRnext^{trad} 
\gets \max\{\gamma_e\|s_k\|,\TRk\} \equiv \max\{\gamma_e\|s_k\|,\gamma_k\|g_k\|\} \in [\gammaLB \|g_{k+1}\|, \gammaUB \|g_{k+1}\|].
$$
In particular, this means that the traditional update is \emph{not} allowed in two scenarios.  The first is when $\max\{\gamma_e\|s_k\|,\gamma_k\|g_k\|\} < \gammaLB \|g_{k+1}\|$.  Since $\gamma_e \geq 1$ and $\gammaLB$ is intended to serve as a lower-bound safeguard (e.g., a typical value might be $10^{-8}$ or smaller), this scenario indicates that the accepted step $s_k$ and gradient $g_k$ are very small in norm compared to the new gradient $g_{k+1}$.  But since $\|g_{k+1}\|$ being large means that the next reduction in $f$ could also be large with a relatively large step, we argue that Update~\ref{up.1} makes sense.
The second scenario in which the traditional update is not allowed is when $\max\{\gamma_e\|s_k\|,\gamma_k\|g_k\|\} > \gammaUB \|g_{k+1}\|$. Since $\gamma_e$ is moderate in size (e.g., a typical value is $\gamma_e = 2$) and $\gammaUB$ is intended to serve as an upper-bound safeguard (e.g., a typical value might be $10^{8}$ or larger), this scenario indicates that the previous radius is significantly larger than the size of the gradient at the new iterate $x_{k+1}$. Since an additional increase in the trust region radius does not seem warranted, we again believe that Update~\ref{up.1} makes sense.

\subsection{A Strategy with a Concise Second-Order Complexity Analysis}\label{sec.second}

In this section, our aim is to provide an upper bound on the maximum number of iterations until, given $(\eps_g,\eps_H) \in \R{}_{++} \times \R{}_{++}$, an iterate $x_k$ satisfies $\|g_k\| \leq \eps_g$ and $\lambda_k \geq -\eps_H$. For this reason, it will be convenient to define the sets
\bequationNN
  \baligned
    \Scal_2(\eps_g,\eps_H) &:= \{k\in\Scal: \|g_k\| > \eps_g \ \text{or} \ |(\lambda_k)_-| > \eps_H\} \\ \text{and}\ \ 
    \Kcal_2(\eps_g,\eps_H) &:= \{k\in\N{}: \|g_k\| > \eps_g \ \text{or} \ |(\lambda_k)_-| > \eps_H\}.
  \ealigned
\eequationNN
Since we are now interested in approximate second-order optimality, and motivated by the decrease in $f$ guaranteed by~\eqref{f-dec} for successful iterations, in this section we adopt the following trust region radius update strategy.

\begin{TRchoice}\label{up.2}
For any $k\in\N{}$ such that Step~\ref{step.delta.1} is reached, in which case $\lambda_k < 0$, set 
$$
\TRk \gets
\begin{cases}
\gamma_k\|g_k\| & \text{if $\|g_k\|^2 \geq |(\lambda_k)_-|^3$,} \\
\gamma_k |(\lambda_k)_-| & \text{if $\|g_k\|^2 < |(\lambda_k)_-|^3$.}
\end{cases}
$$
\end{TRchoice}
The results of this section assume that Update~\ref{up.2} is used.

We first prove an upper bound on the size of the index set $\Scal_2(\eps_g,\eps_H)$.

\begin{lemma}\label{lem.bound-Seps.2}
For any $(\eps_g,\eps_H)\in\R{}_{++}\times\R{}_{++}$, the size of $\Scal_2(\eps_g,\eps_H)$ satisfies
\begin{equation}\label{Seps-bound.2}
  |\Scal_2(\eps_g,\eps_H)| \leq \left\lfloor \left(\frac{f_0 - f_{\inf}}{\kappa_{\min}}\right) \max\left\{\eps_g^{-2},\eps_H^{-3}\right\}\right\rfloor.
\end{equation}
\end{lemma}
\begin{proof}
  Combining Update~\ref{up.2} with Step~\ref{step.delta.2}, it follows that $k \in \Kcal_g$ if and only if $\|g_k\|^2 \geq |(\lambda_k)_-|^3$ while
$k\in\Kcal_H$ if and only if $|(\lambda_k)_-|^3 > \|g_k\|^2$.  Consider $k \in \Scal_2(\eps_g,\eps_H)\cap\Kcal_H$.  By \eqref{f-dec} and $\Kcal\subseteq\Kcal_H$,
  \begin{equation}\label{eq.ftrump1}
    f_k - f_{k+1} \geq \kappa_{\min} |(\lambda_k)_-|^3 \geq \kappa_{\min} \min\{\eps_g^2,\eps_H^3\}.
  \end{equation}
  Now consider $k \in \Scal_2(\eps_g,\eps_H)\cap\Kcal_g$.  By \eqref{f-dec} and $\Kcal\subseteq\Kcal_g$,
  \begin{equation}\label{eq.ftrump2}
    f_k - f_{k+1} \geq \kappa_{\min} \|g_k\|^2 \geq \kappa_{\min} \min\{\eps_g^2,\eps_H^3\}.
  \end{equation}
  Combining \eqref{eq.ftrump1}, \eqref{eq.ftrump2}, Assumption~\ref{ass.main}, and monotonicity of $\{f_k\}$, one finds
  \begin{equation*}
    f_0 - f_{\inf} \geq \sum_{k\in\Scal_2(\eps_g,\eps_H)} (f_k - f_{k+1}) \geq \kappa_{\min} \min\{\eps_g^2,\eps_H^3\} |\Scal_2(\eps_g,\eps_H)|,
  \end{equation*}
  which, after rearrangement, leads to \eqref{Seps-bound.2}.
  

\end{proof}

This leads directly to our second-order complexity result.

\begin{theorem}\label{thm-complexity.2}
For any $(\eps_g,\eps_H)\in\R{}_{++}\times\R{}_{++}$, the size of $|\Kcal_2(\eps_g,\eps_H)|$ satisfies
$$
|\Kcal_2(\eps_g,\eps_H)| 
\leq
\left\lceil\log_{\gamma_c}\(\frac{\gammamin}{\gammaUB}\)\right\rceil \left\lfloor \left(\frac{f_0 - f_{\inf}}{\kappa_{\min}}\right) \max\left\{\eps_g^{-2},\eps_H^{-3}\right\}\right\rfloor
= \Ocal\(\max\big\{\eps_g^{-2},\eps_H^{-3}\big\}\).
$$
\end{theorem}
\begin{proof}
The result follows by combining Lemma~\ref{lem.bound-Seps.2} and Lemma~\ref{lem.bound-consec-U}. \qed
\end{proof}

\subsection{A Strategy with a Fixed Trust Region Radius}\label{sec.ye}

The strategy in Section~\ref{sec.first} offers a complexity of $\Ocal(\eps_g^{-2})$ for driving the norm of the gradient below $\eps_g \in \R{}_{++}$, which is consistent with the complexity of the strategy in Section~\ref{sec.second} when $\Kcal_g = \N{}$ or $\eps_H = \eps_g^{2/3}$.  It is known, however, that certain methods offer a complexity of $\Ocal(\epsilon_g^{-3/2})$; e.g., see \cite{CartGoulToin11b}.  Is it possible to design a trust region radius strategy that leads to this complexity, and what are its advantages and disadvantages compared to the strategies in Sections~\ref{sec.first} and \ref{sec.second} that do not offer the same complexity?  This is the subject of discussion in this subsection.

A trust region method with an $\Ocal(\epsilon_g^{-3/2})$ complexity for achieving approximate first-order stationarity was proposed and analyzed in \cite{CurtRobiSama17}.  This method can be seen, along with that in \cite{CartGoulToin11b}, as a special case of the general framework in \cite{CurtRobiSama17b} for achieving this order complexity.  One can also derive a trust region method with a fixed trust region radius that, with a concise analysis, leads to a $\Ocal(\epsilon_g^{-3/2})$ complexity.  Let us present this analysis now, which follows the lecture notes of Yinyu Ye.\footnote{\href{http://web.stanford.edu/class/msande311/lecture13.pdf}{http://web.stanford.edu/class/msande311/lecture13.pdf}}

Under Assumption~\ref{ass.main} and with $\beta := \thalf L$, it follows from~\cite[Theorem~3.1.6]{ConGT00a} and~\cite[Equation~(1.1)]{CartGoulToin11b} that, for all $(x,s) \in \R{n} \times \R{n}$,
\bsubequations\\[-0.6em]
  \begin{align}
    \|g(x+s) - g(x) - H(x) s\| &\leq \beta\|s\|^2 \label{eq.ye1} \\ \text{and}\ \ 
    f(x+s) - f(x) &\leq g(x)^Ts + \thalf s^TH(x)s + \tfrac13 \beta \|s\|^3. \label{eq.ye2}
  \end{align}
\esubequations
The algorithm that we consider here is one that, for all $k \in \N{}$, sets $\delta_k \gets \sqrt{\eps}/\beta$ and computes $(s_k,\xi_k)$ as a primal-dual solution of \eqref{def.Bmod} satisfying
\bsubequations
  \begin{align}
    g_k + (H_k + \xi_k I)s_k &= 0, \label{eq.dual_feas} \\
    H_k + \xi_k I &\succeq 0, \label{eq.shift} \\
    \xi_k \geq 0,\ \delta_k - \|s_k\| \geq 0,\ \text{and}\ \xi_k(\delta_k - \|s_k\|) &= 0. \label{eq.comp}
  \end{align}
\esubequations
This algorithm, unlike traditional methods, accepts all computed steps.

There are two situations to consider.
\benumerate
  \item[\underline{Case 1:}] If $\xi_k \leq \sqrt{\eps}$, then \eqref{eq.ye1} and \eqref{eq.dual_feas} imply
  \bequationNN
    \baligned
      \|g_{k+1}\|
      &\leq \|g_{k+1} - (g_k + H_ks_k)\| + \|g_k + H_ks_k\| \\
      &\leq \beta \|s_k\|^2 + \xi_k\|s_k\| \leq \beta \delta_k^2 + \xi_k\delta_k = \frac{\eps}{\beta} + \frac{\xi_k\sqrt{\eps}}{\beta} \leq \frac{2\epsilon}{\beta}.
    \ealigned
  \eequationNN
  Next, combining~\cite[page~370]{horn1990matrix} and Assumption~\ref{ass.main} we have $|\lambda_k - \lambda_{k+1}| \leq \|H_k - H_{k+1}\| \leq L\|s_k\|$, which with~\eqref{eq.shift} implies that $\lambda_{k+1} \geq \lambda_k -L\|s_k\| \geq -\xi_k-L\delta_k \geq -\sqrt{\eps} -L\sqrt{\eps}/\beta = -3\sqrt{\eps}$.
  Overall, in iteration $k+1$, one finds
  \bequationNN
    \|g_{k+1}\| \leq \frac{2\eps}{\beta}\ \ \text{and}\ \ \lambda_{k+1} \geq -3\sqrt{\eps}.
  \eequationNN
  \item[\underline{Case 2:}] If $\xi_k > \sqrt{\eps}$, then we know from~\eqref{eq.comp} that $\|s_k\| = \delta_k$, which may then be combined with \eqref{eq.dual_feas} and \eqref{eq.shift} to conclude that
\bequationNN
  g_k^Ts_k + \thalf s_k^TH_ks_k = -\thalf s_k^T(H_k + \xi_k I)s_k - \thalf \xi_k \|s_k\|^2 \leq - \thalf \xi_k \|s_k\|^2 = - \thalf \xi_k \delta_k^2.
\eequationNN
This may, in turn, be used with \eqref{eq.ye2} to obtain
\bequationNN
  f_{k+1} - f_k \leq -\thalf \xi_k \delta_k^2 + \tfrac13 \beta \delta_k^3 = -\frac{\xi_k\eps}{2\beta^2} + \frac{\eps^{3/2}}{3\beta^2} \leq -\frac{\eps^{3/2}}{6\beta^2}.
\eequationNN
  Letting $\Kcal := \{k \in \N{} : \xi_k > \sqrt{\eps}\}$, it follows that
  \bequationNN
    f_0 - f_{\inf} \geq \sum_{k\in\Kcal} (f_k - f_{k+1}) \geq \frac{\epsilon^{3/2}}{6\beta^2} |\Kcal|,
  \eequationNN
  which means that $|\Kcal| = \Ocal(\eps^{-3/2})$.
\eenumerate

Overall, we may conclude from these cases that the number of iterations until $\|g_k\| \leq \eps$ and $\lambda_k \geq -\sqrt{\eps}$ is at most $\Ocal(\eps^{-3/2})$.  One advantage of the strategy employed here is that one might be able to extend this strategy and analysis to situations in which $g_k$ and $H_k$ cannot be computed exactly in each iteration.  However, when $g_k$ and $H_k$ are computable, there are clear costs to achieving this improved complexity.  First, this algorithm requires knowledge of $\beta = \thalf L$, which is not always known in practice.  Second, the algorithm requires exact subproblem solutions.  This restriction might be relaxed using ideas such as in \cite{CurtRobiSama17b}, but one cannot simply employ Cauchy steps as are allowed in the strategies in Sections~\ref{sec.first} and \ref{sec.second}.  Third, the algorithm is dependent on the choice of $\epsilon$, meaning that the desired accuracy needs to be chosen in advance and even early iterations will behave differently depending on the final accuracy desired.  Finally, related to the third point, having the trust region radii depend on $\epsilon$, which is likely to be small, means that the algorithm is likely to take very small steps throughout the optimization process.  This would likely lead to very poor behavior in practice compared to the strategies in Sections~\ref{sec.first} and \ref{sec.second}, which are much less conservative.

\section{Conclusion}\label{sec.conclusion}

We have presented concise complexity analyses of some trust region algorithms. In particular, by choosing the radius to be of the same size as the norm of the gradient, we were able to prove  a complexity result (Theorem~\ref{thm-complexity}) related to first-order stationarity.  For this case, although our method is a special case of the general framework analyzed in~\cite{FanYuan01,grapiglia2015convergence,toint2013nonlinear}, our analysis is simple and highlights the essential aspects needed to obtain the complexity result. Next,
we proposed a new update strategy for the trust region radius that allowed us to to obtain  a complexity result (Theorem~\ref{thm-complexity.2}) for second-order stationarity.  Finally, for comparison purposes, we presented a concise analysis of a trust region method with a fixed radius.

It is unclear how to establish similar complexity results when a traditional trust region radius update is used. For the first-order case, the reason is that, following a successful iteration, it is possible that $\|g_{k+1}\|$ may become too large relative to the trust region radius $\delta_{k+1}$; traditional updating schemes do not appropriately handle this possibility, whereas Update~\ref{up.1} does. For the second-order case, it is similarly possible that $\max\{\|g_{k+1}\|,|(\lambda_{k+1})_-|\}$ may become too large relative to $\delta_{k+1}$;  traditional updating schemes do not account for this, but Update~\ref{up.2} does.

It follows from both Theorem~\ref{thm-complexity} and Theorem~\ref{thm-complexity.2} that $\lim_{k\to\infty} \|g_k\| = 0$. This analysis 
contrasts that of trust region methods that use a traditional radius update strategy, whereby first a liminf result is proved, which is then used to establish the limit result. In the case of Theorem~\ref{thm-complexity}, in which case Update~\ref{up.1} is used, this can be explained by noting that the decrease in the objective function during \emph{all} successful iterations  is proportional to  $\|g_k\|^2$ (cf. proof of Lemma~\ref{lem.bound-Seps}),
which is not always true when a traditional radius update is used.
In the case of Theorem~\ref{thm-complexity.2}, in which case Update~\ref{up.2} is used, this can be explained by noting  that the decrease in the objective function during \emph{all} successful iterations is proportional to 
$\max\{\|g_k\|^2,|(\lambda_k)_-|^3\}$
(cf. proof of Theorem~\ref{lem.bound-Seps.2}),
which is not always true when a traditional radius update is used.

\bibliographystyle{plain}
\bibliography{tr_simple}

\end{document}